\newcommand{\comment}[1]{}

\documentclass[letterpaper,11pt]{article}

\usepackage{amsthm,amsfonts, amssymb, amsmath}
\usepackage[round,authoryear]{natbib}
\usepackage{stmaryrd,url}

\newtheorem{Thm}{Theorem}[section]
\newtheorem{Lem}[Thm]{Lemma}
\newtheorem{Prop}[Thm]{Proposition}
\newtheorem{Coro}[Thm]{Corollary}

\theoremstyle{definition}
\newtheorem{Rem}[Thm]{Remark}
\newtheorem{Def}[Thm]{Definition}
\newtheorem{Example}[Thm]{Example}

\numberwithin{equation}{section}

\newcommand{\ind}{{\bf 1}}
\newcommand{\proba}{\mathbb P}
\newcommand{\esp}{{\mathbb E}}

\newcommand{\defe}{\mathrel{\mathop:}=}

\newcommand{\inv}{^{-1}}

\newcommand{\calA}{{\cal A}}
\newcommand{\calB}{{\cal B}}

\newcommand{\filF}{{\cal F}}
\newcommand{\calG}{{\cal G}}
\newcommand{\calH}{{\cal H}}

\newcommand{\calL}{{\cal L}}

\newcommand{\calN}{{\cal N}}

\def\indn#1{\{#1_n\}_{n\in \mathbb N}}

\def\indz#1{\{#1_k\}_{k\in \mathbb Z}}

\def\indzz#1{\{#1_{i,j}\}_{(i,j)\in\mathbb Z^2}}

\newcommand{\eqnh}{\begin{eqnarray*}}
\newcommand{\eqne}{\end{eqnarray*}}
\newcommand{\eqnhn}{\begin{eqnarray}}
\newcommand{\eqnen}{\end{eqnarray}}
\newcommand{\equh}{\begin{equation}}
\newcommand{\eque}{\end{equation}}

\newcommand{\sumin}{\sum_{i=1}^n}

\def\summ#1#2#3{\sum_{#1 = #2}^{#3}}

\def\sif#1#2{\sum_{#1=#2}^\infty}
\def\sumZ#1{\sum_{#1\in\mathbb Z}}

\def\topp#1{^{(#1)}}

\def\malmosts{\mbox{ almost surely}}

\def\nn#1{{\left\|#1\right\|}}

\def\snn#1{\|#1\|}

\def\bnn#1{\Big\|#1\Big\|}

\def\sabs#1{|#1|}

\def\sccbb#1{\{#1\}}
\def\bpp#1{\Big(#1\Big)}
\def\spp#1{(#1)}

\def\bb#1{\left[#1\right]}
\def\sbb#1{[#1]}
\def\bbb#1{\Big[#1\Big]}
\def\ip#1{\left\langle#1\right\rangle}

\def\vv#1{{\bf #1}}

\def\d{{\rm d}}

\def\eqnhspace{& & \ \ \ \ }

\def\calF{{\mathcal F}}

\def\qmand{\quad\mbox{ and }\quad}
\def\mor{\mbox{ or }}

\def\mwith{\mbox{ with }}
\def\qmwith{\quad\mbox{ with }\quad}
\def\mfa{\mbox{ for all }}

\def\qmif{\quad\mbox{ if }\quad}

\def\adaptF#1{\{#1_t,\filF_t:0\leq t<\infty\}}

\def\weakto{\Rightarrow}

\def\limn{\lim_{n\to\infty}}
\def\limm{\lim_{m\to\infty}}

\def\wt#1{\widetilde{#1}}

\def\ifhead#1#2{\left\{\begin{array}{#1@{\quad\mbox{ if }\quad}#2}}
\def\ifend{\end{array}\right.}

\def\indkd#1{\{#1_k\}_{k\in\mathbb Z^d}}
\def\Zd{{\mathbb Z^d}}
\def\Nd{{\mathbb N^d}}
\def\N{\mathbb N}
\def\OAP{(\Omega,\calA,\proba)}
\def\indij#1{\{#1_{i,j}\}_{(i,j)\in\mathbb Z^2}}
\def\indijn#1{\{#1_{i,j}\}_{(i,j)\in\mathbb N^2}}

\def\B{\mathbb B}
\def\mfa{\mbox{ for all }}

\title{A New Condition for the Invariance Principle for Stationary Random Fields}
\author{Yizao Wang\thanks{The author was partially supported by the NSF grant DMS--0806094 at the University of Michigan.}~ and Michael Woodroofe\\Department of Statistics, University of Michigan}

\begin{document}\sloppy
\maketitle

\begin{abstract}
We establish a central limit theorem and an invariance principle for stationary random fields, with projective-type conditions. Our result is obtained via an $m$-dependent approximation method. As applications, we establish invariance principles for orthomartingales and functionals of linear random fields.\medskip


\end{abstract}

\section{Introduction}
In 1910, \cite{markov1910recherches} proved a central limit theorem for a two-state Markov chain. This initiated one of the longest histories in probability theory, the central limit theorem for stationary processes. One successful approach is the {\it martingale approximation} method, first applied by  \cite{gordin69central} and then developed by many other researchers. Along this line,  \cite{maxwell00central} proved the following result. 
Let $\{X_k\}_{k\in\mathbb Z}$ be a stationary process with $X_k = f\circ T^k$ for all $k\in\mathbb Z$, where $f$ is a measurable function from a probability space $(\Omega,\calA,\proba)$ to $\mathbb R$, and $T$ is a bimeasurable, measure-preserving, one-to-one and onto map on $(\Omega,\calA,\proba)$. Consider 
\equh\label{eq:Snf1}
S_n(f) = \summ k1n f\circ T^k.
\eque
Let $\{\filF_k\}_{k\in\mathbb Z}$ be a filtration on $(\Omega,\calA,\proba)$ such that $T\inv \filF_k = \filF_{k+1}$ for all $k\in\mathbb Z$. Suppose $\int f^2\d\proba <\infty$, $\int f \d\proba = 0$, $f\in\filF_0$ (i.e., the sequence is {\it adapted}) and $\esp(f\mid \bigcap_{k\in\mathbb Z}\filF_k) = 0$. Maxwell and Woodroofe proved that, if
\equh\label{cond:MW00}
\sif k1\frac{\snn{\esp(S_k(f)\mid\filF_0)}_2}{k^{3/2}}<\infty\,,
\eque
then $\sigma^2 = \limn \esp(S_n^2)/n$ exists, and
\[
\frac{S_n(f)}{\sqrt n}\weakto\calN(0,\sigma^2)\,.
\]
Here `$\weakto$' denotes the weak convergence of the random variables (convergence in distribution), and the $L^2$ norm $\nn\cdot_{2}$ is with respect to the measure $\proba$.
Note that~\eqref{cond:MW00} is implied by
\equh\label{cond:MW00'}
\sif k1\frac{\snn{\esp(f\circ T^k\mid\filF_0)}_2}{k^{1/2}}<\infty\,.
\eque

Condition~\eqref{cond:MW00} is referred to as the Maxwell--Woodroofe condition. 
Later on,  \cite{peligrad05new} showed that~\eqref{cond:MW00} also implies the invariance principle. Indeed, let $\{\mathbb B(t)\}_{t\in[0,1]}$ denote the standard Brownian motion.
Then,~\eqref{cond:MW00} implies
\[
\frac{S_{\left\lfloor n\cdot\right\rfloor}(f)}{\sqrt n}\weakto\sigma \mathbb B(\cdot)
\]
where $\left\lfloor x\right\rfloor$ denotes the largest integer smaller or equal to $x\in\mathbb R$ and `$\weakto$' is understood as the weak convergence in $C[0,1]$.
Furthermore, Peligrad and Utev showed that~\eqref{cond:MW00} is the best possible (among conditions that only restrict the size of $\nn{\esp( S_n(f)\mid\filF_0)}_2$). See also   \cite{dedecker07weak} and  \cite{durieu08comparison} for comparisons of Conditions~\eqref{cond:MW00} and~\eqref{cond:MW00'} with other sufficient conditions for central limit theorems.  
For non-adapted sequences (i.e., $f\notin\filF_0$), a similar condition guaranteeing the invariance principle is established by   \cite{volny07nonadapted}. 
Other important references on central limit theorems by martingale approximation include 
 \cite{gordin78central},  \cite{kipnis86central},   \cite{woodroofe92central},  \cite{wu04martingale},  \cite{dedecker07weak},  \cite{peligrad07maximal}, among others, and  \cite{merlevede06recent} for a survey. 
The martingale approximation can also be applied to establish invariance principle for empirical processes, see for example   \cite{wu03empirical,wu08empirical}, and for random walks in random environment, see for example   \cite{rassoulagha05almost, rassoulagha07quenched}.

In this paper, we establish a central limit theorem and an invariance principle for stationary multiparameter random fields. We briefly mention a few results in the literature.   \cite{bolthausen82central},  \cite{goldie86central} and  \cite{bradley89caution} studied this problem under suitable mixing conditions.   \cite{basu79functional},  \cite{nahapetian95billingsley}, and  \cite{poghosyan98invariance} considered the problem for {\it multiparameter martingales}. 
Another important result is due to  \cite{dedecker98central,dedecker01exponential}, whose approach was based on an adaptation of the {\it Lindeberg method}. As a particular case,  \cite{cheng06central}  established a central limit theorem for {\it functionals of linear random fields}, based on a lexicographically ordered martingale approximation.  

Here, we aim at establishing the so-called {\it projective-type} conditions such that the central limit theorem and invariance principle hold. Such conditions, often involving conditional expectations as in~\eqref{cond:MW00} and~\eqref{cond:MW00'}, have recently drawn much attentions in central limit theorems for stationary sequences (see e.g.~\cite{dedecker07weak}). In particular, such conditions are easy to verify when applying such results to stochastic processes from statistics and econometrics (see e.g.~\cite{wu11asymptotic}). However, central limit theorems for stationary random fields based on projective conditions have been much less explored. 

This problem is not a simple extension of a one-dimensional problem to a high-dimensional one. An important reason is that, the main technique for establishing central limit theorems with projective conditions in one dimension, the {\it martingale approximation} approach, does not apply to (high-dimensional) random fields as successfully as to (one-dimensional) stochastic processes. This obstacle has been known among researchers for more than 30 years. For example,  \cite{bolthausen82central} remarked that `Gordin uses an approximation by martingales, but his method appears difficult to generalizes to dimensions $\geq 2$.'

Our result, with a condition similar to~\eqref{cond:MW00'}, is a first attempt of extending central limit theorems with projective-type conditions to the multiparameter stationary random fields. The result is obtained by a different approximation approach, namely, approximation by $m$-dependent random fields.

To state our main result, we start with some notations. We consider a {\it product probability space} $\OAP$, i.e., a $\Zd$-indexed-product of i.i.d.~probability spaces in form of
\[
\OAP \equiv (\mathbb R^\Zd,\calB^\Zd,P^\Zd)\,.
\]
Write $\epsilon_k(\omega) = \omega_k$, for all $\omega\in\mathbb R^\Zd$ and $k\in\Zd$. Then, $\indkd\epsilon$ are i.i.d.~random variables with distribution $P$. 
On such a space, we define the {\it natural filtration} $\indkd\filF$ by
\equh\label{eq:filFk}
\filF_k \defe \sigma\{\epsilon_l: l\preceq k, l\in\Zd\}, \mfa k\in\Zd\,.
\eque
Here and in the sequel, for all vector $x\in\mathbb R^d$, we write $x = (x_1,\dots,x_d)$ and for all $l,k\in\mathbb R^d$, let $l\preceq k$ stand for $l_i\leq k_i, i = 1,\dots,d$.

We focus on mean-zero stationary random fields, defined on a product probability space.
Let $\indkd T$ denote the group of shift operators on $\mathbb R^\Zd$ with $(T_k\omega)_l = \omega_{k+l}$, for all $k,l\in\Zd\,, \omega\in\mathbb R^{\mathbb Z^d}$. Then, we consider random fields in form of 
\[
\indkd {f\circ T}\,, \mbox{ or equivalently } \{f(\epsilon_{k+l}:l\in\mathbb Z^d)\}_{k\in\mathbb Z^d}\,,
\]
where 
$f$ is in the class $\calL_0^p = \{f\in L^p(\filF_\infty), \int f\d\proba = 0\}, p\geq 2$, with $\filF_{\infty} = \bigvee_{k\in\mathbb Z^d}\filF_k$. 

Throughout this paper, we consider a sequence $\indn V$ of finite rectangular subsets of $\Zd$, in form of
\equh\label{eq:Vn}
V_n = \prod_{i=1}^d\{1,\dots,m_i\topp n\}\subset\mathbb N^d\,, \mfa n\in\mathbb N\,,
\eque
with $m_i\topp n$ increasing to infinity as $n\to\infty$ for all $i = 1,\dots,d$.
Let 
\equh\label{eq:Snf}
S_n(f)\equiv S(V_n,f) = \sum_{k\in V_n}f\circ T_k
\eque
denote the partial sums with respect to $V_n$. 
Moreover, write for $t\in[0,1]$, $V_n(t) = \prod_{i=1}^d[0,m_i\topp nt]\subset \mathbb R^d$ and $R_k = \prod_{i=1}^d(k_i-1,k_i]\subset\mathbb R^d$ for all $k\in\Zd$. 

We write also
\equh\label{eq:Bntf}
B_{n,t}(f)\equiv B_{V_n,t}(f) = \sum_{{k\in\Nd}}\lambda(V_n(t)\cap R_k)f\circ T_k\,,
\eque
where $\lambda$ is the Lebesgue measure on $\mathbb R^d$, and consider the weak convergence in the space $C[0,1]^d$, the space of continuous functions on $[0,1]^d$, equipped with the uniform metric. Recall that the standard $d$-parameter Brownian sheet on $[0,1]^d$, denoted by $\{\B(t)\}_{t\in[0,1]^d}$, is a mean-zero Gaussian random field with covariance $\esp(\B(s)\B(t)) = \prod_{i=1}^d\min(s_i,t_i), s,t\in[0,1]^d$. Write $\vv 0 = (0,\dots,0), \vv 1 = (1,\dots,1)\in\Zd$. In parallel to~\eqref{cond:MW00'}, our projective-type condition involves the following term:
\equh\label{cond:wtDeltad}
\wt\Delta_{d,p}(f) 
\defe \sum_{k\in\mathbb N^d} \frac{\snn{\esp(f\circ T_{k}\mid\filF_{\vv 1})}_p}{\prod_{i=1}^dk_i^{1/2}}\,.
\eque
Our main result is the following. 
\begin{Thm}\label{thm:1}
Consider a product probability space described above. If
 $f\in\calL_0^2$, $f\in\filF_{\vv 0}$ and
$\wt\Delta_{d,2}(f)<\infty$,
then 
\[
\sigma^2 = \limn \frac{\esp(S_n(f)^2)}{|V_n|} <\infty
\]
exists and 
\[
\frac{S_n(f)}{|V_n|^{1/2}} \weakto\calN(0,\sigma^2)\,.
\] 
In addition, if $f\in\calL_0^p$ and $\wt\Delta_{d,p}(f)<\infty$ for some $p>2$, then
\equh\label{eq:IP}
\frac{B_{n,\cdot}(f)}{|V_n|^{1/2}}\weakto \sigma\mathbb B(\cdot)
\eque
in $C[0,1]^d$.
\end{Thm}
For the sake of simplicity, we will prove Theorem~\ref{thm:1} in the case $d = 2$ in Sections~\ref{sec:CLT} and~\ref{sec:IP}. 
 
We develop two applications of the main result. First, we obtain a central limit theorem for {\it orthomartingales}, a special class of multiparameter martingale (see e.g.~\cite{khoshnevisan02multiparameter}), defined on a product probability space. To the best of our knowledge, this result is more general than existing central limit theorems for multiparameter martingales (\cite{basu79functional},  \cite{nahapetian95billingsley} and   \cite{poghosyan98invariance}), on which we provide a detailed discussion in Section~\ref{sec:orthomartingales}. In particular, we demonstrate that one should not expect a central limit theorem even for general orthomartingales, without extra conditions on the structure of the underlying probability space. 

Second, we obtain an invariance principle of functionals of stationary causal linear random fields in Section~\ref{sec:LRF}. This result extends the work of  \cite{wu02central} in the one-dimensional case. Another central limit theorem for functional of stationary linear random fields has recently been developed by  \cite{cheng06central}, following the approach of   \cite{ho97limit} and  \cite{cheng05asymptotic} in the one-dimensional case.  We provide simple examples where our condition is weaker.

\begin{Rem}
After we finished this work,  \cite{elmachkouri11central} obtained a central limit theorem and an invariance principle for stationary random fields, in the similar spirit as ours. They took also an $m$-approximation approach, based on the {\it physical dependence measure} introduced by  \cite{wu05nonlinear}.
Their results are more general, in the sense that they established invariance principle for random fields indexed by arbitrary sets instead of rectangle ones. Their conditions are not directly comparable to ours. However, in the application to functionals of linear random fields, their condition on the coefficients is weaker (see Remark~\ref{rem:EVW}). 
\end{Rem}
The paper is organized as follows. In Section~\ref{sec:prelim} we provide preliminary results on $m$-dependent approximation. We establish the central limit theorem in Section~\ref{sec:CLT} and then the invariance principle in Section~\ref{sec:IP}. Sections~\ref{sec:orthomartingales} and~\ref{sec:LRF} are devoted to the applications to orthomartingales and functionals of stationary linear random fields, respectively. In Section~\ref{sec:momentInequality}, we prove a moment inequality, which plays a crucial role in proving our limit results. Some other auxiliary proofs are given in Section~\ref{sec:proofs}.

\section{$m$-Dependent Approximation}\label{sec:prelim}
We describe the general procedure of $m$-dependent approximation in this section. In this section, we do not assume any structure on the underlying probability space, nor the filtration structure. Instead, we simply assume $f\in L^2_0 = \{f\in L^2\OAP, \int f\d\proba = 0\}$, and $\indkd T$ is an Abelian group of bimeasurable, measure-preserving, one-to-one and onto maps on $\OAP$. 

The notion of $m$-dependence was  introduced by  \cite{hoeffding48central}. We say a random variable $f$ is {\it $m$-dependent}, if $f\circ T_k, f\circ T_l$ are independent whenever $|k-l|_\infty \defe\max_{i= 1,\dots,d}|k_i-l_i|>m$. 
The following result on the asymptotic normality of sums of $m$-dependent random variables is a consequence of  \cite{bolthausen82central} (see also  \cite{rosen69note}). Recall $\indn V$ given in~\eqref{eq:Vn}.

\begin{Thm}\label{thm:rosen}
Suppose $f_m\in L^2_0$ is $m$-dependent and. Write
\equh\label{eq:sigmamVn}
\sigma^2_m = \sum_{k\in\Zd}\esp[f_m(f_m\circ T_k)]\,.
\eque
Then,
\[
\frac{S_n(f_m)}{|V_n|^{1/2}} \weakto \calN(0,\sigma_m^2)\,.
\]
\end{Thm}
Now, consider the function $f\in L_0^2(\proba)$ and define
\equh\label{eq:plus}
\nn f_{V,+} = \limsup_{n\to\infty}\frac{\nn{S_n(f)}_2}{|V_n|^{1/2}}\,.
\eque
We refer to the pseudo norm defined by $\nn\cdot_{V,+}$ as the {\it plus-norm}.
\begin{Lem}\label{lem:CLT}
Suppose $f,f_1,f_2,\dots\in L_0^2(\proba)$ and $f_m$ is $m$-dependent for all $m\in\mathbb N$.
If
\equh\label{eq:2}
\limm\nn{f-f_m}_{V,+} = 0\,,
\eque
then
\equh\label{eq:1}
\limm \sigma_m = \limm \nn{f_m}_{V,+} =:\sigma<\infty
\eque
exists, and 
\equh\label{eq:CLT}
\frac{S_n(f)}{|V_n|^{1/2}}\weakto\calN(0,\sigma^2)\,.
\eque
\end{Lem}
\begin{proof}
It suffices to prove~\eqref{eq:1}.
We will show that $\{\sigma_m^2\}_{m\in\mathbb N}$ forms a Cauchy sequence in $\mathbb R_+$. Observe that since $f_m$ is $m$-dependent with zero mean, 
\[
\sigma_m = \limn \frac{\snn {S_n(f_{m})}_2}{|V_n|^{1/2}} \,.
\]
It then follows that
\eqnh
|\sigma_{m_1}-\sigma_{m_2}| &  \leq & {\limsup_{n\to\infty}\frac{{\snn{S_n(f_{m_1}-f_{m_2})}_2}}{|V_n|^{1/2}}}\\
& \leq & \snn{f_{m_1}-f}_{V,+}+\snn{f_{m_2}-f}_{V,+}\,,
\eqne
which can be made arbitrarily small by taking $m_1, m_2$ large enough. We have thus shown that $\indn{\sigma^2}$ is a Cauchy sequence in $\mathbb R_+$.
\end{proof}
\begin{Rem}
The idea of establishing the central limit theorem by controlling the quantity $\snn{f-f_m}_{V,+}$ dates back to  \cite{gordin69central}, where $f_m$ was selected from a different subspace. In the one-dimensional case, when $V_n = \{1,\dots,n\}$,  \cite{zhao08martingale} named $\nn\cdot_{V,+}$ the plus-norm, and established a necessary and sufficient condition for the martingale approximation, in term of the plus-norm. See  \cite{peligrad10conditional} and  \cite{gordin11functional} for improvements and more discussions on such conditions.
\end{Rem}
In the next section, we will establish conditions, under which~\eqref{eq:2} holds.
\section{A Central Limit Theorem}\label{sec:CLT}
From this section on, we will focus on stationary multiparameter random fields, defined on product probability spaces. On Such a space, any integrable function has a natural $L^2$-approximation by $m$-dependent functions, and there is a natural commuting filtration.

For the sake of simplicity, we consider only the 2-parameter random fields in the sequel and simply say `random fields' for short. We will prove a central limit theorem here and then an invariance principle in the next section. The argument, however, can be generalized easily to $d$-parameter random fields, and the result has been stated in Theorem~\ref{thm:1}.

We start with a product probability space with i.i.d.~random variables $\indij\epsilon$. Recall that $\indij T$ are the group of shift operators on $\mathbb R^{\mathbb Z^2}$ and write $\filF_{\infty,\infty} = \sigma(\epsilon_{i,j}:(i,j)\in\mathbb Z^2)$. We focus on the class of functions $\calL_0^p = \{f\in L^p(\filF_{\infty,\infty}): \esp f = 0\}, p\geq 2$. For all measurable function $f\in\calL_0^2$, define, for all $m\in\mathbb N$, 
\equh\label{eq:fm}
f_m \defe \esp\spp{f| \filF_{\ip m}}\qmwith \filF_{\ip m} = \sigma(\epsilon_j:j\in\{-m,\dots,m\}^2)\,.
\eque
Clearly, $f_m\in\calL_0^2$, $\nn {f-f_m}_2\to 0$ as $m\to\infty$ and $\indij{f_m\circ T}$ are $m$-dependent functions.

Now, recall the natural filtration $\indij\filF$ defined by $\filF_{k,l} = \sigma(\epsilon_{i,j}:i\leq k, j\leq l)$. This is a 2-parameter filtration, i.e., 
\equh\label{eq:multiFiltration}
\filF_{i,j}\subset\filF_{k,l}\qmif i\leq k, j\leq l\,.
\eque
Also, 
\equh\label{eq:nested}
T_{-i,-j}\filF_{k,l} = \filF_{k+i,l+j}\,,\forall (i,j),(k,l)\in\mathbb Z^2\,.
\eque
Moreover, the notion of commuting filtration is of importance to us.
\begin{Def}\label{def:commuting}
A filtration $\indij\filF$ is {\it commuting}, if for all $\filF_{k,l}$-measurable bounded random variable $Y$, $\esp(Y|\filF_{i,j}) = \esp(Y|\filF_{i\wedge k,j\wedge l})$.
\end{Def}
Since $\{\epsilon_{k,l}\}_{(k,l)\in\mathbb Z^2}$ are independent random variables, $\indij\filF$ is {commuting} (see Proposition~\ref{prop:commuting} in Section~\ref{sec:proofs}).  This implies that the {\it marginal filtrations}
\equh\label{eq:marginal}
\filF_{i,\infty} = \bigvee_{j\geq 0}\filF_{i,j} \qmand \filF_{\infty,j} = \bigvee_{i\geq 0}\filF_{i,j}
\eque
are commuting, in the sense that for all $Y\in L^1(\proba)$, 
\equh\label{eq:marginalCommuting}
\esp[\esp(Y|\filF_{i,\infty})|\filF_{\infty,j}] = \esp[\esp(Y|\filF_{\infty,j})|\filF_{i,\infty}] = \esp(Y|\filF_{i,j})\,.
\eque
For more details on the commuting filtration, see  \cite{khoshnevisan02multiparameter}.

For all $\filF_{0,0}$-measurable function $f\in\calL_0^2$, write
\equh\label{eq:Smn}
S_{m,n}(f) = \summ i1m\summ j1{n} f\circ T_{i,j}\,.
\eque
Thanks to the commuting structure of the filtration, applying twice the maximal inequality in  \cite{peligrad07maximal}, we can prove the following moment inequality with $p\geq 2$:
\equh\label{eq:PUW}
\snn{S_{m,n}(f)}_p \leq Cm^{1/2}n^{1/2}\Delta_{(m,n),p}(f)
\eque
with
\[
\Delta_{(m,n),p}(f) = \summ k1{m}\summ l1{n}\frac{\snn{\esp(S_{k,l}(f)\mid\filF_{1,1})}_p}{k^{3/2}l^{3/2}}\,.
\]
In fact, we will prove a stronger inequality without the assumptions of product probability space and the $\filF_{0,0}$-measurability of $f$. See Section~\ref{sec:momentInequality}, Proposition~\ref{prop:momentInequality} and Corollary~\ref{coro:1}.
 
Recall that
\equh\label{eq:wtDelta}
\wt\Delta_{2,p}(f) 
= \sif k1\sif l1 \frac{\snn{\esp(f\circ T_{k,l}\mid\filF_{1,1})}_p}{k^{1/2}l^{1/2}}\,.
\eque
Now, we can prove the following central limit theorem for adapted stationary random fields.
\begin{Thm}\label{thm:CLT}
Consider the product probability space discussed above. Let $\indn V$ be as in~\eqref{eq:Vn} with $d=2$. Suppose $f\in\calL_0^2$, $f\in\filF_{0,0}$, and define $f_m$ as in~\eqref{eq:fm}. If $\wt\Delta_{2,2}(f)<\infty$,
then \[
\limm\nn {f-f_m}_{V,+} = 0\,.
\]
Therefore, $\sigma \defe \limm\nn{f_m}_{V,+}<\infty$ exist and $S_n(f)/|V_n|^{1/2}\weakto\calN(0,\sigma^2)$.
\end{Thm}
\begin{proof}
The second part follows immediately from Lemma~\ref{lem:CLT}.
It suffices to prove $\nn{f-f_m}_{V,+}\to 0$ as $m\to\infty$. 
First, by the fact that 
\[
\snn{\esp(S_{k,l}(f)\mid\filF_{1,1})}_2\leq \summ i1k\summ j1l\snn{\esp\spp{f\circ T_{i,j}\mid\filF_{1,1}}}_2
\]
and Fubini's theorem, we have $\Delta_{(\infty,\infty),2}(f)\leq 9\wt\Delta_{2,2}(f)$. 
So, by~\eqref{eq:plus} and~\eqref{eq:PUW}, it suffices to show 
\equh\label{eq:DCT0}
\wt\Delta_{2,2}(f-f_m)= \sif k1\sif l1\frac{\snn{\esp[(f-f_m)\circ T_{k,l}\mid\filF_{1,1}]}_2}{k^{1/2}l^{1/2}} \to 0
\eque
as $m\to\infty$.
Clearly, the summand in~\eqref{eq:DCT0} converges to 0 for each $k,l$ fixed, since~\eqref{eq:fm} implies $\nn{f-f_m}_2 \to 0$ as $m\to\infty$ and $\snn{\esp[(f-f_m)\circ T_{k,l}\mid\filF_{1,1}]}_2 \leq \snn{f-f_m}_2$. Moreover,
observe that,
\eqnh
  \esp\spp{f_m\circ T_{k,l}\mid\filF_{1,1}}  & = & \esp\sbb{\esp\spp{f\circ T_{k,l}\mid T_{-k,-l}(\filF_{\left\langle m\right\rangle})}\mid\filF_{1,1}}\\
& = & \esp\sbb{\esp\spp{f\circ T_{k,l}\mid\filF_{1,1}}\mid T_{-k,-l}(\filF_{\left\langle m\right\rangle})}\,,
\eqne
where in the second equality we can exchange the order of conditional expectations by the definitions of $\filF_{1,1}$ and $T_{-k,-l}(\filF_{\left\langle m\right\rangle})$ (see Proposition~\ref{prop:commuting} in Section~\ref{sec:proofs} for a detailed treatment). Therefore,
\eqnh
& & \snn{\esp\sbb{(f-f_m)\circ T_{k,l}\mid\filF_{1,1}}}_2 \\
\eqnhspace \quad\quad\quad\leq
  \snn{\esp\spp{f\circ T_{k,l}\mid\filF_{1,1}}}_2 + \snn{\esp\spp{f_m\circ T_{k,l}\mid\filF_{1,1}}}_2 \\
\eqnhspace \quad\quad\quad \leq 2\snn{\esp\spp{f\circ T_{k,l}\mid\filF_{1,1}}}_2\,.
\eqne
Then, the condition $\wt\Delta_{2,2}(f)<\infty$ combined with the dominated convergence theorem yields~\eqref{eq:DCT0}. The proof is thus completed.
\end{proof}
\begin{Rem}An `extension' of Maxwell--Woodroofe condition~\eqref{cond:MW00} to high dimension remains an open problem. Namely if we replace $\wt\Delta_{2,2}(f)<\infty$ by $\Delta_{(\infty,\infty),2}(f)<\infty$ in Theorem~\ref{thm:CLT}, do we have the same conclusion? The latter condition is significantly weaker than the former one. 
\end{Rem}

\section{An Invariance Principle}\label{sec:IP}
Recall the space $C[0,1]^2$ and the 2-parameter Brownian sheet $\{\B(t)\}_{t\in[0,1]^2}$.
\begin{Thm}\label{thm:IP}
Under the assumptions in Theorem~\ref{thm:CLT}, suppose in addition that $f\in\calL_0^p$ and $\wt\Delta_{2,p}(f)<\infty$ for some $p>2$. Write $B_{n,t}(f)$ as in~\eqref{eq:Bntf} with $d=2$. Then, 
\[
\frac{B_{n,\cdot}(f)}{|V_n|^{1/2}}\weakto\sigma\mathbb B(\cdot)\,,
\]
where `$\ \weakto$' stands for weak convergence of probability measures on $C[0,1]^2$.
\end{Thm}
\begin{proof}
It suffices to show that the finite-dimensional distributions converge, and $\{B_{n,t}(f)/|V_n|^{1/2}\}_{t\in[0,1]^2}$ is tight. 

We first show that, for all $\wt t = (t\topp1,\dots,t\topp k)\subset[0,1]^2$, 
\equh\label{eq:fdd}
\bpp{\frac{B_{n,t\topp1}(f)}{|V_n|^{1/2}},\cdots,\frac{B_{n,t\topp k}(f)}{|V_n|^{1/2}}}\weakto \sigma(\mathbb B(t\topp1),\cdots,\mathbb B(t\topp k)) =:\sigma\wt{\mathbb B}_{\wt t}\,.
\eque
Consider  the $m$-dependent function $f_m$ defined in~\eqref{eq:fm}. Then, the convergence of the finite-dimensional distributions~\eqref{eq:fdd} with $f$ replaced by $f_m$ follows from the invariance principle of $m$-dependent random fields (see e.g.~\cite{shashkin03invariance}). Furthermore, by Theorem~\ref{thm:CLT}, $\wt\Delta_{2,2}(f)\leq\wt\Delta_{2,p}(f)<\infty$, so that $\nn{f-f_m}_{V,+}\to 0$ as $m\to\infty$, and therefore, letting $\wt B_{n,\wt t}(f)/|V_n|^{1/2}$ denote the left-hand side of~\eqref{eq:fdd}, $\wt B_{n,\wt t}(f_m-f)/|V_n|^{1/2}\to (0,\dots,0)\in\mathbb R^k$ in probability. The convergence of the finite-dimensional distribution~\eqref{eq:fdd} follows.

Now, we prove the tightness of $\{B_{n,t}(f)\}_{t\in[0,1]^2}$. Fix $n$ and consider 
\[
V_n = \{1,\dots,n_1\}\times\{1,\dots,n_2\}\,.
\]
Write $B_{n,t} \equiv B_{n,t}(f)$ and $S_{m,n} \equiv S_{m,n}(f)$ for short. For all $0\leq r_1<s_1\leq 1, 0\leq r_2<s_2\leq 1$, set,
\[
B_n((r_1,s_1]\times(r_2,s_2]) \defe  B_{n,(s_1,s_2)} -  B_{n,(r_1,s_2)} - B_{n,(s_1,r_2)} +  B_{n,(r_1,r_2)}\,.
\]
We will show that there exists a constant $C$, independent of $n, r_1,r_2,s_1$ and $s_2$, such that
\equh\label{eq:bound}
(n_1n_2)^{-1/2}\nn{B_n((r_1,s_1]\times(r_2,s_2])}_p\leq C\sqrt{(s_1-r_1)(s_2-r_2)}\wt\Delta_{2,p}(f)\,.
\eque
Inequality~\eqref{eq:bound} implies the tightness, by  \cite{nagai74simple}, Theorem 1.

Now, we prove~\eqref{eq:bound} to complete the proof. From now on, the constant $C$ may change from line to line.
Write $m_i = \left\lfloor n_is_i\right\rfloor - \left\lfloor n_ir_i\right\rfloor, i = 1,2$. If $m_i\geq 2, i = 1,2$, then 
\eqnhn
& & \snn{B_n((r_1,s_1]\times(r_2,s_2])}_p  \nonumber\\
& &\leq   \nn{S_{m_1,m_2}}_p  + 2\snn{S_{m_1,1}}_p + 2\snn{S_{1,m_2}}_p+ 4\snn{S_{1,1}}_p\nonumber\\
& &\leq C(m_1m_2)^{1/2}\wt\Delta_{2,p}(f)\label{eq:bound1}
\eqnen
for some constant $C$, by~\eqref{eq:PUW}.
Note that $m_i\geq 2$ also implies $n_i(s_i-r_i)>1$. Therefore, $m_i\leq n_i(s_i-r_i)  + 1<2n_i(s_i-r_i)$, and~\eqref{eq:bound1} can be bounded by $C(n_1n_2)^{1/2}[(s_1-r_1)(s_2-r_2)]^{1/2}\wt\Delta_{2,p}(f)$, which yields~\eqref{eq:bound}. 

In the case $m_1<2$ or $m_2<2$, to obtain~\eqref{eq:bound} requires more careful analysis. We only show the case when $m_1 = 1, m_2\geq 2$, as the proof for the other cases are similar. Suppose that $m_1 = 1$ and we exclude the case $n_1r_1 = \lfloor n_1r_1\rfloor = \lceil n_1r_1\rceil$ (it is easy to see that this case can be eventually controlled by continuity). Then, we have  $n_1r_1<\lceil n_1r_1\rceil = \lfloor n_1s_1\rfloor\leq  n_1s_1$. Then,
\begin{multline*}
\snn{B_n((r_1,s_1]\times(r_2,s_2])}_p \\
 \leq n_1(s_1-r_1)( \nn{S_{1,m_2}}_p  + 2\snn{S_{1,1}}_p)
 \leq Cn_1(s_1-r_1)m_2^{1/2}\wt\Delta_{2,p}(f)\,.
\end{multline*}
Observe that $m_1 = 1$ also implies $n_1(s_1-r_1)\in(0,2)$. If $n_1(s_1-r_1)\leq 1$, then $n_1(s_1-r_1)\leq [n_1(s_1-r_1)]^{1/2}$. If $n_1(s_1-r_1)\in (1,2)$, then $n_1(s_1-r_1)< \sqrt 2[n_1(s_1-r_1)]^{1/2}$. It then follows that~\eqref{eq:bound} still holds.
\end{proof}
\begin{Rem}
To prove the invariance principle of stationary random fields, most of the results require a finite moment of order strictly larger than 2. See for example \cite{berkes81strong},  \cite{goldie86variance} and  \cite{dedecker01exponential}. This is in contrast to the one-dimensional case, where the invariance principle can be established with finite second moment assumption.

To the best of our knowledge, the only invariance principle so far for stationary random fields that assumes finite second moment is due to  \cite{shashkin03invariance}, where the random fields are assumed to be $BL(\theta)$-dependent (including $m$-dependent stationary random fields). In general the $BL(\theta)$-dependence is difficult to check. Besides,  \cite{basu79functional} proved an invariance principle for martingale difference random fields with finite second moment assumption, but they have stringent conditions on the filtration (see Remark~\ref{rem:orthomartingale} below). In our case, it remains an open problem: whether $\wt\Delta_{2,2}(f)<\infty$ implies the invariance principle. See also a similar conjecture in  \cite{dedecker01exponential}, Remark 1.
\end{Rem}
\section{Orthomartingales}\label{sec:orthomartingales}
The central limit theorems and invariance principles for multiparameter martingales are more difficult to establish than in the one-dimensional case. This is due to the complex structure of multiparameter martingales. We will focus on orthomartingales first and establish an invariance principle, and then compare the results on other types of multiparameter martingales. 

The idea of orthomartingales are due to R.~Cairoli and J.~B.~Walsh. See e.g.~references in  \cite{khoshnevisan02multiparameter}, which also provides a nice introduction to the materials. For the sake of simplicity, we suppose $d=2$. 
Consider a probability space $\OAP$ and recall the definition of 2-parameter filtration~\eqref{eq:multiFiltration}. We restrict ourselves to the filtration indexed by $\mathbb N^2$.

\begin{Def}
Given a commuting 2-parameter filtration $\indijn\filF$ on $\OAP$, we say a family of random variables $\indijn M$ is a {\it 2-parameter orthomartingale} on $\OAP$, with respect to $\indijn\filF$, if for all $(i,j)\in\mathbb N^2$, $M_{i,j}$ is $\filF_{i,j}$-measurable, and $\esp(M_{i+1,j}\mid\filF_{i,\infty}) = \esp(M_{i,j+1}\mid\filF_{\infty,j}) = M_{i,j}$, almost surely. 
\end{Def}
In our case, for $\filF_{0,0}$-measurable function $f\in\calL_0^2$, $M_{m,n} = S_{m,n}(f)$ as in~\eqref{eq:Smn} yields a 2-parameter orthomartingale, if 
\equh\label{eq:orthomartingale}
\esp(f\circ T_{i+1,j}\mid\filF_{i,\infty}) = \esp(f\circ T_{i,j+1}\mid\filF_{\infty,j}) = 0\malmosts,
\eque
for all $(i,j)\in\N^2$. In this case, we say $\indijn{f\circ T}$ are {\it 2-parameter orthomartingale differences}.
\begin{Rem}
In our case, $\indijn M$ is also a {\it 2-parameter martingale} in the normal sense, i.e., $\esp(M_{i,j}\mid\filF_{k,l}) = M_{i\wedge k,j\wedge l}$, almost surely. Indeed, 
\[
\esp(M_{i,j}\mid\filF_{k,l}) = \esp[\esp(M_{i,j}\mid\filF_{k,\infty})\mid\filF_{\infty,l}] = \esp(M_{i\wedge k,j}\mid\filF_{\infty,l}) = M_{i\wedge k,j\wedge l}\,.
\]
In general, however, the converse is not true, i.e., multiparameter martingales are not necessarily orthomartingales (see e.g.~\cite{khoshnevisan02multiparameter} p.~33). The two notions are equivalent, when the filtration is commuting (see e.g.~\cite{khoshnevisan02multiparameter}, Chapter I, Theorem 3.5.1).
\end{Rem}

\begin{Thm}\label{thm:orthomartingale}
Consider a product probability space $\OAP$ with a natural filtration $\indijn\filF$. Suppose $f\in\calL_0^2$ and $f\in\filF_{0,0}$. If $\indijn{f\circ T}$ are 2-parameter orthomartingale differences, i.e.,~\eqref{eq:orthomartingale} holds, then $\sigma^2 = \limn \esp(S_n(f)^2)/|V_n|^2 <\infty$ exists, and 
\[
\frac{S_n(f)}{|V_n|^{1/2}}\weakto\sigma\calN(0,1)\,.
\]
In addition, if $f\in\calL_0^p$ for some $p>2$, then the invariance principle~\eqref{eq:IP} holds.
\end{Thm}
\begin{proof}
Observe that,~\eqref{eq:orthomartingale} implies $\esp(f\circ T_{i,j}\mid\filF_{1,1}) = 0$ if $i>1$ or $j>1$. Then, for $f\in\calL_0^p, p\geq 2$,
\[
\wt\Delta_{\infty,p}(f) = \snn{\esp(f\circ T_{1,1}\mid\filF_{1,1})}_p = \nn f_p<\infty\,.
\]
The result then follows immediately from Theorem~\ref{thm:1}. Note that, the argument holds for general $d$-parameter orthomartingales ($d\geq 2$) defined in  \cite{khoshnevisan02multiparameter}.
\end{proof}
\begin{Rem}\label{rem:orthomartingale}
Our result is more general than   \cite{basu79functional},   \cite{nahapetian95billingsley} and   \cite{poghosyan98invariance} in the following sense. Let be $\indij\epsilon$ be i.i.d.~random variables. In  \cite{nahapetian95billingsley}, the central limit theorem was established for the so-called {\it martingale-difference random fields} $\indijn M$ with $M_{i,j} = \summ k1i\summ l1j D_{k,l}$, such that 
\[
\esp[D_{i,j}\mid\sigma(\epsilon_{k,l}:(k,l)\in\mathbb Z^2, (k,l)\neq (i,j))] = 0\,,\mfa (i,j)\in\mathbb N^2\,.
\]
In  \cite{basu79functional} and  \cite{poghosyan98invariance}, the authors considered the multiparameter martingales $\indijn M$ with respect to the filtration defined by
\[
\wt\filF_{i,j} = \sigma(\epsilon_{k,l}: k\leq i \mor l\leq j)\,.
\]
It is easy to see, in both cases above, their assumptions are stronger, in the sense that they imply that $\indijn M$ is an orthomartingale, with the natural filtration $\indijn\filF$~\eqref{eq:filFk}. On the other hand, however, the results in (  \cite{basu79functional,poghosyan98invariance}) only assume that $\indij\epsilon$ is a stationary random field, which is weaker than our assumption. 
\end{Rem}
\begin{Rem}
By assumption, the $\sigma$-algebra of $\{T_{i,j}\}_{(i,j)\in\mathbb Z^2}$-invariant sets is $\proba$ trivial. Therefore, our results are restricted to ergodic random fields, and exclude the following simple case:
\[
X_{i,j} = Y\epsilon_{i,j}, (i,j)\in\mathbb Z^2\,,
\]
where $Y$ is a random variable independent of $\{\epsilon_{i,j}\}_{(i,j)\in\mathbb Z^2}$. Clearly, if $\epsilon_{0,0}$ has zero mean and finite variance $\sigma^2$, then
\[
\frac1n\summ i1n\summ j1n X_{i,j}\weakto YZ\,,
\]
where $Z\sim\calN(0,\sigma^2)$ is independent of $Y$.
For central limit theorems on non-ergodic random fields, see for example  \cite{dedecker98central,dedecker01exponential}. 
\end{Rem}

At last, we point out that the product structure of the probability space plays an important role. We provide an example of an orthomartingale with a different underlying probability structure. In this case, the limit behavior is quite different from the case that we studied so far.
\begin{Example}\label{example:productRWs}
Suppose $\indz\epsilon$ and $\indz{\eta}$ are two families of i.i.d.~random variables. Define $\calG_i = \sigma(\epsilon_j:j\leq i)$ and $\calH_i = \sigma(\eta_j:j\leq i)$ for all $i\in\N$. Then, $\calG = \indn\calG$ and $\calH = \indn\calH$ are two filtrations. 

Now, let $\indn Y$ and $\indn Z$ be two arbitrary martingales with stationary increment with respect to the filtration $\calG$ and $\calH$, respectively. Suppose $Y_n = \sumin D_i, Z_n = \sumin E_i$, where $\indn D$ and $\indn E$ are stationary martingale differences. Then, $\{D_iE_j\}_{(i,j)\in\N^2}$ is a stationary random fields and 
\[
M_{m,n} \defe  \summ i1m\summ j1n D_iE_j = Y_mZ_n
\]
is an orthomartingale with respect to the filtration $\{ \calG_i\vee\calH_j\}_{(i,j)\in\mathbb N^2}$. Clearly, 
\[
\frac{M_{n,n}}n = \frac{Y_n}{\sqrt n}\frac{Z_n}{\sqrt n}\weakto \calN(0,\sigma_Y^2)\times\calN(0,\sigma_Z^2)\,,
\]
where the limit is the distribution of the product of two independent normal random variables (a Gaussian chaos). That is, $M_{n,n}/n$ has asymptotically non-normal distribution.

One can also define $\wt M_{m,n} = Y_m+Z_n$, which again gives an orthomartingale, and $\{D_i+E_j\}_{(i,j)\in\mathbb N^2}$ is the corresponding stationary random field. This time, one can show that 
\[
\frac{\wt M_{n,n}}{\sqrt n} = \frac{Y_n}{\sqrt n} + \frac{Z_n}{\sqrt n}\weakto \calN(0,\sigma_Y^2+\sigma_Z^2)\,.
\]
Here, the limit is a normal distribution, but the normalizing sequence is $\sqrt n$ instead of $n$.
\end{Example}
This example demonstrates that for general orthomartingales, to obtain a central limit theorem one must assume extra conditions on the structure of the underlying probability space. For the structure mentioned above, there is no $m$-dependent approximation for the random fields. Indeed, the example corresponds to the sample space $\Omega = (\mathbb R^{\mathbb Z},\mathbb R^{\mathbb Z})$ with $[T_{k,l}(\epsilon,\eta)]_{i,j} = (\epsilon_{i+k},\eta_{j+l})$, and if we define $f_m$ similarly as in~\eqref{eq:fm} with
\[
\calF_{\ip m} \defe \sigma(\epsilon_i,\eta_j:-m\leq i,j\leq m)\,,
\]
then $f$ and $f\circ T_{k,l}$ are independent, if and only if $\min(k,l)>m$. That is, the dependence can be very strong, along the horizontal (the vertical resp.) direction of the random field.

\section{Stationary Causal Linear Random Fields}\label{sec:LRF}
We establish a central limit theorem for functionals of stationary causal linear random fields.  We focus on $d=2$. Consider a stationary linear random field $\{Z_{i,j}\}_{(i,j)\in\mathbb Z^2}$ defined by
\equh\label{eq:Zij}
Z_{i,j} = \sumZ r\sumZ sa_{r,s}\epsilon_{i-r,j-s} = \sumZ r\sumZ sa_{i-r,j-s}\epsilon_{r,s}\,,
\eque
where coefficients $\{a_{i,j}\}_{(i,j)\in\mathbb Z^2}$ satisfy $\sum_{(i,j)\in\mathbb Z^2}a_{i,j}^2<\infty$, and $\{\epsilon_{i,j}\}_{(i,j)\in\mathbb Z^2}$ are i.i.d.~random variables with zero mean and finite variance as before. We restrict ourselves to {\it causal} linear random fields, i.e., $a_{i,j} = 0$ unless $i\geq 0$ and $j\geq 0$. They are also referred to be {\it adapted} to the filtration $\{\filF_{i,j}\}_{(i,j)\in\mathbb Z^2}$. 

Now, consider the random fields $\{f\circ T_{k,l}\}_{(k,l)\in\mathbb Z^2}$ with a more specific form $f = K(\{Z_{i,j}\}_h^{0,0})$, where $h$ is a fixed strictly positive integer, $K$ is a measurable function from $\mathbb R^{h^2}$ to $\mathbb R$ and for all $(k,l)\in\mathbb Z^2$,
\[
\{Z_{i,j}\}_h^{k,l} \defe \{Z_{i,j}:k-h+1\leq i\leq k, l-h+1\leq j\leq l\}
\]
is viewed as a random vector in $\mathbb R^{h^2}$ with covariates lexicographically ordered. In the sequel, the same definition applies similarly to $\{x_{i,j}\}_h^{k,l}$, given $\{x_{i,j}\}_{(i,j)\in\mathbb Z^2}$.
Assume that 
\equh\label{eq:K}
\esp K(\{Z_{i,j}\}_h^{0,0}) = 0\qmand \esp K^p(\{Z_{i,j}\}_h^{0,0}) <\infty
\eque
for some $p\geq 2$.
In this way, 
\equh\label{eq:Xkl}
f\circ T_{k,l} = K(\{Z_{i,j}\}_h^{k,l})\,.
\eque
The model~\eqref{eq:Xkl} is a natural extension of the functionals of causal linear processes considered by  \cite{wu02central}.

Next, we introduce a few notations similar to  \cite{ho97limit} and  \cite{wu02central}. Here, our ultimate goal is to translate Condition~\eqref{eq:wtDelta} into a condition on the regularity of $K$ and the summability of $\{a_{i,j}\}_{(i,j)\in\mathbb Z^2}$. For all $(i,j)\in\mathbb Z^2$, let
\equh\label{eq:Gammaij}
\Gamma(i,j) = \{(r,s)\in\mathbb Z^2:r\leq i,s\leq j\}\,,
\eque
and write
\eqnhn
Z_{i,j} & = & \sum_{(r,s)\in\Gamma(i,j)}a_{i-r,j-s}\epsilon_{r,s}\nonumber\\
& = & \sum_{(r,s)\in\Gamma(i,j)\setminus\Gamma(1,1)}a_{i-r,j-s}\epsilon_{r,s} + \sum_{(r,s)\in\Gamma(1,1)}a_{i-r,j-s}\epsilon_{r,s}\nonumber\\
& =: & Z_{i,j,+}+Z_{i,j,-}\,.\label{eq:Zij+-}
\eqnen
Write $W_{k,l,-} = \{Z_{i,j,-}\}_h^{k,l}$ and define, for all $(k,l)\in\mathbb Z^2$,
\[
K_{k,l}(\{x_{i,j}\}_h^{k,l}) = \esp K(\{Z_{i,j,+} + x_{i,j}\}_h^{k,l})\,.
\]
In this way,
\equh\label{eq:Kn1}
\esp\spp{f\circ T_{k,l}\mid\filF_{1,1}} = K_{k,l}(\{Z_{i,j,-}\}_h^{k,l}) =: K_{k,l}(W_{k,l,-})\,.
\eque
Plugging~\eqref{eq:Kn1} into~\eqref{eq:wtDelta}, we obtain a central limit theorem for functionals of stationary causal linear random fields.
\begin{Thm}
Consider the functionals of stationary causal linear random fields~\eqref{eq:Xkl}. If Conditions~\eqref{eq:K} hold and
\equh\label{cond:1}\sif k1\sif l1\frac{\nn {K_{k,l}(W_{k,l,-})}_p}{k^{1/2}l^{1/2}}<\infty\,,
\eque
for $p=2$, then $\sigma^2 = \limn\esp(S_n^2)/n^2 <\infty$ exists and $S_n/|V_n|^{1/2}\weakto \calN(0,\sigma^2)$. If the conditions hold with $p>2$, then the invariance principle~\eqref{eq:IP} holds.
\end{Thm}
Next, we will provide conditions on $K$ and $\{a_{i,j}\}_{(i,j)\in\mathbb Z^2}$ such that~\eqref{cond:1} holds. For all $\Lambda\subset \mathbb Z^2$, write 
\equh\label{eq:ZLambda}
Z_\Lambda = \sum_{(i,j)\in\Lambda}a_{i,j}\epsilon_{-i,-j}\qmand A_\Lambda = \sum_{(i,j)\in\Lambda} a_{i,j}^2\,.
\eque
In particular, our conditions involves summations of $a_{i,j}$ over the following type of regions:
\[
\Lambda(k,l) \defe \{(i,j)\in\mathbb Z^2: i\geq k, j\geq l\}\,, (k,l)\in\mathbb Z^2.
\]
For the sake of simplicity, we write $A_{k,l} \equiv A_{\Lambda(k,l)}$. The following  lemma is a simple extension of Lemma 2, part (b) in  \cite{wu02central}. 
\begin{Lem}\label{lem:K}
Suppose that there exist $\alpha,\beta\in\mathbb R$ such that $0<\alpha\leq 1\leq \beta<\infty$ and $\esp(|\epsilon|^{2\beta})<\infty$. If
\equh\label{eq:M}
\esp M_{\alpha,\beta}^2(W_{1,1})<\infty\mwith M_{\alpha,\beta}(x) = \sup_{\substack{y\in\mathbb R^{h^2},\ y\neq x}}\frac{|K(x)-K(y)|}{|x-y|^\alpha + |x-y|^\beta}\,,
\eque
then, for all $p\geq 2$,
\equh\label{eq:Kkl2}
\snn{K_{k,l}(W_{k,l,-})}_p = O(A_{k+1-h,l+1-h}^{\alpha/2})\,.
\eque
\end{Lem}
The proof is deferred to Section~\ref{sec:proofs}.
Consequently, Condition~\eqref{cond:1} can be replaced by specific ones on $A_{k,l}$.
\begin{Coro}
Assume there exist $\alpha,\beta\in\mathbb R$ as in Lemma~\ref{lem:K}. 
Consider the functionals of stationary linear random fields in form of~\eqref{eq:Xkl}. Suppose Condition~\eqref{eq:M} holds and
\equh\label{cond:2}
\sif k1\sif l1\frac{A_{k+1-h,l+1-h}^{\alpha/2}}{k^{1/2}l^{1/2}}<\infty\,.
\eque
If $\esp(|\epsilon|^p)<\infty$ and~\eqref{eq:K} hold with $p=2$, then $S_n/n\weakto\calN(0,\sigma^2)$ with some $\sigma<\infty$. If $\esp(|\epsilon|^p)<\infty$ and~\eqref{eq:K} holds with $p>2$, then the invariance principle~\eqref{eq:IP} holds.
\end{Coro}
We compare our Condition~\eqref{cond:2} on the summability of $\{a_{i,j}\}_{(i,j)\in\mathbb Z^2}$, and the one considered by  \cite{cheng06central}. They only established central limit theorems for functionals of  stationary linear random fields, so we restrict to the case $p=2$.   \cite{cheng06central} assumed
\equh\label{eq:aij1/2}
\sif i0\sif j0 |a_{i,j}|^{1/2}<\infty\,,
\eque
and provided different regularity conditions on $K$. Namely, 
\[
\sup_{\Lambda\subset\mathbb Z^2}\esp K^2(x + Z_\Lambda)<\infty
\]
for all $x\in\mathbb R$ with $Z_\Lambda$ defined in~\eqref{eq:ZLambda}, and that for any two independent random variables $X$ and $Y$ with $\esp(K^2(X) + K^2(Y) + K^2(X+Y))<M<\infty$, 
\equh\label{eq:K@cheng06}
\esp[(K(X+Y)-K(X))^2]\leq C[\esp (Y^2)]^\gamma
\eque
for some $\gamma\geq 1/2$. 
In general,   \cite{cheng06central}'s condition and ours on the regularity $K$ are not comparable and thus have different range of applications. Below, we focus on the simple case that $h = 1$ and $K$ is Lipschitz, covered by both works. This corresponds to $\alpha = \beta = 1$ in~\eqref{eq:M} and $\gamma = 1$ in~\eqref{eq:K@cheng06}.
In the following two examples, our Condition~\eqref{cond:2} is weaker than Condition~\eqref{eq:aij1/2}. 
\begin{Example}\label{example:1}
Consider $a_{i,j} = (i+j+1)^{-q}$ for all $i,j\geq 0$ and some $q>1$. Then, $A = \sif i0\sif j0 a_{i,j}^2<\infty$ and
\[
A_{k,l} = \sif j1 j(k+l+j)^{-2q} = O((k+l)^{2-2q})\,.
\]
Then~\eqref{cond:2} is bounded by, up to a multiplicative constant,
\[
\sif k1\sif l1 \frac{(k+l)^{1-q}}{k^{1/2}l^{1/2}} < \sif k1 \frac{k^{(1-q)/2}}{k^{1/2}}\sif l1 \frac{l^{(1-q)/2}}{l^{1/2}} \leq \bpp{\sif k1 k^{-q/2}}^2\,.
\]
Therefore, Condition~\eqref{cond:2} requires $q>2$.
In this case, Condition~\eqref{eq:aij1/2} requires $q>4$.
\end{Example}
\begin{Example}\label{example:2}
Consider $a_{i,j} = (i+1)^{-q}(j+1)^{-q}$, for all $i,j\geq 0$ for some $q>1$. Then, $A = \sif i0\sif j0 a_{i,j}^2<\infty$ and
\equh
A_{k,l} = \sif ik\sif jl a_{i,j}^2 =  O(k^{-(2q-1)}l^{-(2q-1)})\,.
\eque
One can thus check that Condition~\eqref{cond:2} requires $q>3/2$ while Condition~\eqref{eq:aij1/2} requires $q>2$.
\end{Example}
\begin{Rem}\label{rem:EVW}
For the central limit theorem for functionals of linear random fields, the weakest condition known is due to  \cite{elmachkouri11central} (Example 1 and Theorem 1), who showed that it suffices to require $K$ to be Lipschitz and 
\[
\sum_{i,j}|a_{i,j}|<\infty\,.
\]
Furthermore, their result and the one by  \cite{cheng06central} do not assume the linear random field to be causal.
\end{Rem}
\section{A Moment Inequality}\label{sec:momentInequality}
We establish a moment inequality for stationary 2-parameter random fields on general probability spaces, without assuming the product structure.
We first review the Peligrad--Utev inequality, a maximal $L^p$-inequality in dimension one, with $p\geq 2$.  Recall the partial summation in~\eqref{eq:Snf1} and the related probability space. Let $C$ denote a constant that may change from line to line.
It is known that for all $f\in L^p(\filF_\infty)$ with $\esp (f\mid\filF_{-\infty}) = 0$, 
\begin{multline}\label{eq:Volny}
\bnn{\max_{1\leq k\leq n}|S_k(f)|}_p \leq Cn^{1/2}\Bigg(\nn{\esp(f\mid\filF_0)}_p + \nn {f - \esp(f\mid\filF_0)}_p  \\
+ \summ k1n\frac{\nn{\esp(S_k(f)\mid\filF_{0})}_p}{k^{3/2}} +\summ k1n\frac{\nn{S_k(f) - \esp(S_k(f)\mid\filF_{k})}_p}{k^{3/2}}\Bigg)\,.
\end{multline}
The inequality above was first established for adapted stationary sequences in  \cite{peligrad05new} and then extended to $L^p$-inequality for $p\geq 2$ in  \cite{peligrad07maximal}. The  case $p\in(1,2)$ was addressed by  \cite{wu08moderate}. The non-adapted case for $p\geq 2$ was addressed by  \cite{volny07nonadapted}.

For the sake of simplicity, we simplify the bound in~\eqref{eq:Volny} by regrouping the summations. Observe that $\snn{\esp(S_k(f)\mid\filF_0)}_p\leq \snn{\esp(S_k(f)\mid\filF_1)}_p$, $\nn{\esp(f\mid\filF_0)}_p = \nn{\esp(S_1(f)\mid\filF_1)}_p$ and $\snn{f-\esp(f\mid\filF_0)}_p = \snn{S_1(f)-\esp(S_1(f)\mid\filF_1)}_p$. Thus, we obtain
\begin{multline}\label{eq:Volny2}
\bnn{\max_{1\leq k\leq n}|S_k(f)|}_p \\\leq Cn^{1/2}\Bigg(\summ k1{n}\frac{\nn{\esp(S_k(f)\mid\filF_{1})}_p}{k^{3/2}} +\summ k1{n}\frac{\nn{S_k(f) - \esp(S_k(f)\mid\filF_{k})}_p}{k^{3/2}}\Bigg)\,.
\end{multline}

Now, consider a general probability space $\OAP$, and suppose there exists a commuting 2-parameter filtration $\indzz\filF$, and an Abelian group of bimeasurable, measure-preserving, one-to-one and onto maps $\indij T$ on $\OAP$, such that~\eqref{eq:nested} holds. Define $\filF_{\infty,\infty} = \bigvee_{(i,j)\in\mathbb Z^2}\filF_{i,j}$, $\filF_{-\infty,\infty} = \bigcap_{i\in\mathbb Z}\filF_{i,\infty}$ and $\filF_{\infty,-\infty} = \bigcap_{j\in\mathbb Z}\filF_{\infty,j}$. Note that when $\OAP$ is a product probability space, then $\filF_{-\infty,\infty}$ and $\filF_{\infty,-\infty}$ are trivial, by Kolmogorov's zero-one law. 

Recall the definition of $S_{m,n}(f)$ in~\eqref{eq:Smn}. 
Given $f$, write $S_{m,n} \equiv S_{m,n}(f)$ for the sake of simplicity. 
\begin{Prop}\label{prop:momentInequality}
Consider $\OAP$, $\indij T$ and $\indij\filF$ described as above. Suppose $p\geq 2$, $f\in L^p(\filF_{\infty,\infty})$ and $\esp(f\mid\filF_{-\infty,\infty}) = \esp(f\mid\filF_{\infty,-\infty}) = 0$. Then,
\[
\snn{S_{m,n}}_p \leq Cm^{1/2}n^{1/2}\summ k1{m}\summ l1{n}\frac{d_{k,l}(f)}{k^{3/2}l^{3/2}}
\]
with
\eqnh
d_{k,l}(f) & = & \snn{\esp(S_{k,l}\mid\filF_{1,1})}_p\\
& & + \snn{\esp(S_{k,l}\mid\filF_{1,\infty}) - \esp(S_{k,l}\mid\filF_{1,l})}_p \\
& & + \snn{\esp(S_{k,l}\mid\filF_{\infty,1}) - \esp(S_{k,l}\mid\filF_{k,1})}_p\\
& & + \snn{S_{k,l} - \esp(S_{k,l}\mid\filF_{k,\infty}) - \esp(S_{k,l}\mid\filF_{\infty,l}) + \esp(S_{k,l}\mid\filF_{k,l})}_p\,.
\eqne
\end{Prop}
\begin{Coro}\label{coro:1}
Suppose the assumptions in Proposition~\ref{prop:momentInequality} hold.
\begin{itemize}
\item [(i)] If $f\in \filF_{0,0}$, then
\[
\snn{S_{m,n}(f)}_p \leq Cm^{1/2}n^{1/2}\summ k1{m}\summ l1{n}\frac{\snn{\esp(S_{k,l}(f)\mid\filF_{1,1})}_p}{k^{3/2}l^{3/2}}\,.
\]
\item[(ii)] If $\{f\circ T_{i,j}\}_{(i,j)\in\mathbb Z^2}$ are two-dimensional martingale differences, in the sense that $f\in L^p(\filF_{0,0})$ and $\esp(f\mid\filF_{0,-1}) = \esp(f\mid\filF_{-1,0}) = 0$, then
\[
\snn{S_{m,n}(f)}_p\leq Cm^{1/2}n^{1/2}\nn f_p\,.
\]
\end{itemize}
\end{Coro}
The proof of Corollary~\ref{coro:1} is trivial. We only remark that the second case recovers the Burkholder's inequality for multiparameter martingale differences established in  \cite{fazekas05burkholder}.
\begin{proof}[Proof of Proposition~\ref{prop:momentInequality}]
Fix $f$. Define $\wt S_{0,n} = \summ j1n f\circ T_{0,j}$. Clearly,
\equh\label{eq:S0n}
S_{m,n} = \summ i1m\summ j1n f\circ T_{i,j} = \summ i1m\bpp{\summ j1n f\circ T_{0,j}}\circ T_{i,0} =  \summ i1{m}\wt S_{0,n}\circ T_{i,0}\,.
\eque
Fix $n$. Observe that $\esp \wt S_{0,n}=0$ and $\wt S_{0,n}\circ T_{i,0}$ is a stationary sequence. Furthermore, $\{\filF_{i,\infty}\}_{i\in\mathbb Z}$ is a filtration, $T_{i,0}\inv\filF_{j,\infty} = T_{-i,0}\filF_{j,\infty} = \filF_{i+j,\infty}$ and $\esp(\wt S_{0,n}\mid\filF_{-\infty,\infty}) = 0$. Therefore, we can apply the Peligrad--Utev inequality~\eqref{eq:Volny2} and obtain
\eqnhn
\snn{S_{m,n}}_p & \leq & Cm^{1/2}\Big(\summ k1{m}k^{-3/2}\underbrace{\snn{\esp(S_{k,n}\mid\filF_{1,\infty})}_p}_{\Lambda_1}\nonumber\\
& & \quad\quad\quad + \summ k1{m}k^{-3/2}\underbrace{\snn{S_{k,n}-\esp(S_{k,n}\mid\filF_{k,\infty})}_p}_{\Lambda_2}\Big)\,.\label{eq:11}
\eqnen
We first deal with $\Lambda_1$. Define $\wt S_{m,0} = \summ i1m f\circ T_{i,0}$. Similarly as in~\eqref{eq:S0n}, $S_{k,n} = \summ j1n \wt S_{k,0}\circ T_{0,j}$, and 
\eqnhn
\esp(S_{k,n}\mid\filF_{1,\infty}) 
& = & \summ j1{n}\esp\spp{\wt S_{k,0}\circ T_{0,j}\mid\filF_{1,\infty}}\nonumber\\
& = & \summ j1{n}\esp\spp{\wt S_{k,0}\circ T_{0,j}\mid T_{0,-j}(\filF_{1,\infty})}\nonumber\,,
\eqnen
where in the last equality we used the fact that $T_{0,j}(\filF_{i,\infty}) = \filF_{i,\infty}$, for all $i,j\in\mathbb Z$. Now, by the identify $\esp(f\mid\filF)\circ T = \esp(f\circ T\mid T\inv(\filF))$, we have
\equh\label{eq:Lambda1}
\esp(S_{k,n}\mid\filF_{1,\infty}) = \summ j1{n}\esp(\wt S_{k,0}\mid\filF_{1,\infty})\circ T_{0,j}\,.
\eque
Observe that~\eqref{eq:Lambda1} is again a summation in the form of~\eqref{eq:Snf1}. Then, applying the Peligrad--Utev inequality~\eqref{eq:Volny2} again, we obtain
\eqnh
\Lambda_1 & \leq & Cn^{1/2}\Big(\summ l1{n}l^{-3/2}\snn{\esp[\esp(S_{k,l}\mid\filF_{1,\infty})\mid\filF_{\infty,1}]}_p\\
& & \quad\quad\quad + \summ l1{n}l^{-3/2}\snn{\esp(S_{k,l}\mid\filF_{1,\infty})-\esp[\esp(S_{k,l}\mid\filF_{1,\infty})\mid\filF_{\infty,l}]}_p\Big)\,.
\eqne
By the commuting property of the marginal filtrations~\eqref{eq:marginalCommuting}, the above inequality becomes
\eqnhn
\Lambda_1 & \leq & Cn^{1/2}\Big(\summ l1{n}l^{-3/2}\snn{\esp(S_{k,l}\mid\filF_{1,1})}_p\nonumber\\
& & \quad\quad\quad + \summ l1{n}l^{-3/2}\snn{\esp(S_{k,l}\mid\filF_{1,\infty})-\esp(S_{k,l}\mid\filF_{1,l})}_p\Big)\,.\label{eq:12}
\eqnen
Similarly, one can show
\eqnhn
\Lambda_2 & = & \bnn{\summ j1{n}\sbb{S_{k,0}-\esp(S_{k,0}\mid\filF_{k,\infty})}\circ T_{0,j}}_p\nonumber\\
& \leq & Cn^{1/2}\Big(\summ l1{n}l^{-3/2}\snn{\esp(S_{k,l}\mid\filF_{\infty,1})-\esp(S_{k,l}\mid\filF_{k,1})}_p\nonumber\\
& & \quad\quad\quad + \summ l1{n}l^{-3/2}\| S_{k,l}-\esp(S_{k,l}\mid\filF_{k,\infty})\nonumber\\
& & \quad\quad\quad\quad\quad\quad - \esp(S_{k,l}\mid\filF_{\infty,l})+\esp(S_{k,l}\mid\filF_{k,l})\|_p\Big)\,.\label{eq:13}
\eqnen
Combining~\eqref{eq:11},~\eqref{eq:12} and~\eqref{eq:13}, we have thus proved Proposition~\ref{prop:momentInequality}.
\end{proof}
\section{Auxiliary Proofs}\label{sec:proofs}
For arbitrary $\sigma$-fields $\filF,\calG$, let $\filF\vee\calG$ denote the smallest $\sigma$-field that contains $\filF$ and $\calG$.
\begin{Prop}\label{prop:commuting}
Let $(\Omega,\calB,\proba)$ be a probability space and let $\filF,\calG,\calH$ be mutually independent sub-$\sigma$-fields of $\calB$. Then, for all random variable $X\in\calB$, $\esp |X|<\infty$, we have
\equh\label{eq:commuting}
\esp\bb{\esp(X\mid \filF\vee\calG)\mid\calG\vee\calH} = \esp(X\mid\calG)\mbox{ a.s.}
\eque
\end{Prop}
Proposition~\ref{prop:commuting} is closely related to the notion of conditional independence (see e.g.~\cite{chow78probability}, Chapter 7.3). Namely, provided a probability space $(\Omega,\filF,\proba)$, and sub-$\sigma$-fields $\calG_1,\calG_2$ and $\calG_3$ of $\filF$, $\calG_1$ and $\calG_2$ are said to be {\it conditionally independent} given $\calG_3$, if for all $A_1\in\calG_1, A_2\in\calG_2$, $\proba(A_1\cap A_2\mid\calG_3) = \proba(A_1\mid\calG_3)\proba(A_2\mid\calG_3)$ almost surely. 
\begin{proof}[Proof of Proposition~\ref{prop:commuting}]
First, we show that $\filF\vee\calG$ and $\calG\vee\calH$ are conditionally independent, given $\calG$. By Theorem 7.3.1 (ii) in  \cite{chow78probability}, it is equivalent to show, for all $F\in\filF, G\in\calG$, $\proba(F\cap G\mid\calG\vee\calH) = \proba(F\cap G\mid\calG)$ almost surely. This is true since
\[
\proba(F\cap G\mid\calG\vee\calH) = \ind_G\esp(\ind_F\mid\calG\vee\calH) = \ind_G\esp(\ind_F\mid\calG) = \proba(F\cap G\mid\calG)\mbox{ a.s.}
\]
Next, by Theorem 7.3.1 (iv) in  \cite{chow78probability}, the conditional independence obtained above yields $\esp(X\mid\calG\vee\calH)  = \esp(X\mid\calG)$ almost surely, for all $X\in\filF\vee\calG$, $\esp|X|<\infty$. Replacing $X$ by $\esp(X\mid\filF\vee\calG)$, we have thus proved~\eqref{eq:commuting}.
\end{proof}
\begin{proof}[Proof of Lemma~\ref{lem:K}]
Write $W_{k,l} = \{Z_{i,j}\}_h^{k,l}$. Define (and recall that) $W_{k,l,\pm} = \{Z_{i,j,\pm}\}_h^{k,l}$. Let $\wt W_{k,l,-}$ be a copy of $W_{k,l,-}$, independent of $W_{k,l,\pm}$. Set $\wt W_{k,l}\defe W_{k,l,+}+\wt W_{k,l,-}$. 

Recall $K_{k,l}(W_{k,l,-}) = \esp(K(W_{k,l})\mid\filF_{1,1})$ in~\eqref{eq:Kn1}.
Observe that by~\eqref{eq:Zij+-}, $W_{k,l,-}\in\filF_{1,1}$, and $W_{k,l,+}, \wt W_{k,l,-}$ are independent of $\filF_{1,1}$. Therefore, $\esp(K(\wt W_{k,l})\mid\filF_{1,1}) = \esp(K(\wt W_{k,l})) = 0$, and
\eqnh
|K_{k,l}(W_{k,l,-})| 
& = & \sabs{\esp(K(W_{k,l}) - K(\wt W_{k,l})\mid\filF_{1,1})}\\
& \leq & {\esp(|K(W_{k,l}) - K(\wt W_{k,l})|\mid\filF_{1,1})}\,.
\eqne
Observe that by~\eqref{eq:M},
\[
|K(W_{k,l}) - K(\wt W_{k,l})|\leq M_{\alpha,\beta}(\wt W_{k,l})(|W_{k,l,-}-\wt W_{k,l,-}|^\alpha + |W_{k,l,-}-\wt W_{k,l,-}|^\beta)\,.
\]
Write $U_{k,l} = W_{k,l,-}-\wt W_{k,l,-}$. By Cauchy--Schwartz's inequality, and noting that $\esp(|M_{\alpha,\beta}(\wt W_{k,l})|^2\mid\filF_{1,1}) = \snn{M_{\alpha,\beta}(\wt W_{k,l})}_2^2 = \snn{M_{\alpha,\beta}(\wt W_{1,1})}_2^2$, we have
\[
|K_{k,l}(W_{k,l,-})|\leq \snn{M_{\alpha,\beta}(\wt W_{1,1})}_2 \sccbb{\esp\sbb{(|U_{k,l}|^\alpha+|U_{k,l}|^\beta)^2\mid\filF_{1,1}}}^{1/2}\,,
\]
whence, for $p\geq 2$,
\eqnhn
\snn{K_{k,l}(W_{k,l,-})}_p & \leq & \snn{M_{\alpha,\beta}(\wt W_{1,1})}_2\snn{|U_{k,l}|^\alpha+|Y_{k,l}|^\beta}_p \nonumber\\
& \leq & \snn{M_{\alpha,\beta}(\wt W_{1,1})}_2(\snn{|U_{k,l}|^\alpha}_p+\snn{|Y_{k,l}|^\beta}_p)\,.
\label{eq:Kkl3}
\eqnen
Finally, since for all $\gamma>0$ and $n\in\mathbb N$, there exists a constant $C(\gamma,n)>0$ such that and for all vector $w=(w_1,\dots,w_n)\in\mathbb R^n$, 
\[
|w|^{2\gamma} = \bpp{\sum_{i=1}^nw_i^2}^\gamma \leq C(\gamma,n)\bpp{\sum_{i=1}^n w_i^{2\gamma}}\,,
\]
it follows that for all $\gamma>0$,
\eqnhn
\esp(|U_{k,l}|^{2\gamma}) & = & \esp(|W_{k,l,-}-\wt W_{k,l,-}|^{2\gamma}) = \esp(|\{Z_{i,j,-} - \wt Z_{i,j,-}\}_h^{k,l}|^{2\gamma})\nonumber\\
& = & O\bbb{\esp\bpp{\sum_{\substack{k-h<i\leq k\\l-h<j\leq l}}(Z_{i,j,-}-\wt Z_{i,j,-})^{2\gamma}}}\,. \nonumber
\eqnen
By  \cite{wu02central}, Lemma 4, under the notation~\eqref{eq:ZLambda}, $\esp(|\epsilon|^{2\vee2\gamma})<\infty$ implies that for all $\Lambda\subset \mathbb Z^2$, $\esp(|Z_\Lambda|^{2\gamma}) \leq CA_\Lambda^\gamma$ for some universal constant $C$. It then follows that $\esp(|U_{k,l}|^{2\gamma}) = O(A_{k+1-h,l+1-h}^\gamma)$.
Consequently,~\eqref{eq:Kkl3} yields
\eqnh
\nn{K_{k,l}(W_{k,l,-})}_p & \leq & \nn{M_{\alpha,\beta}(W_{1,1})}_2\bbb{O(A_{k+1-h,l+1-h}^{\alpha/2}) + O(A_{k+1-h,l+1-h}^{\beta/2})} \\
& = & O(A_{k+1-h,l+1-h}^{\alpha/2})\,.
\eqne
The proof is thus completed.
\end{proof}

\noindent{\bf Acknowledgment} The authors thank Stilian Stoev for many constructive and helpful discussions, and in particular his suggestion of considering the $m$-dependent approximation approach.

\def\cprime{$'$} \def\polhk#1{\setbox0=\hbox{#1}{\ooalign{\hidewidth
  \lower1.5ex\hbox{`}\hidewidth\crcr\unhbox0}}}

\end{document}